\documentclass[reqno]{amsart}
\usepackage{amssymb,setspace}
\usepackage[hmarginratio=1:1,textwidth=17cm,height=23cm]{geometry}%
\numberwithin{equation}{section}
\usepackage{ifpdf}
\ifpdf
  \usepackage[hyperindex,pagebackref]{hyperref}%
\else
  \expandafter\ifx\csname dvipdfm\endcsname\relax
    \usepackage[hypertex,hyperindex,pagebackref]{hyperref}
  \else
    \usepackage[dvipdfm,hyperindex,pagebackref]{hyperref}
  \fi
\fi
\allowdisplaybreaks[4]
\theoremstyle{plain}
\newtheorem{thm}{Theorem}[section]

\theoremstyle{definition}
\newtheorem{dfn}{Definition}[section]
\theoremstyle{remark}
\newtheorem{rem}{Remark}[section]

\DeclareMathOperator{\td}{d}

\begin{document}

\title[Explicit formulae for Euler numbers and polynomials]
{Explicit formulae for computing Euler polynomials in terms of Stirling numbers of the second kind}

\author[B.-N. Guo]{Bai-Ni Guo}
\address[Guo]{School of Mathematics and Informatics, Henan Polytechnic University, Jiaozuo City, Henan Province, 454010, China}
\email{\href{mailto: B.-N. Guo <bai.ni.guo@gmail.com>}{bai.ni.guo@gmail.com}, \href{mailto: B.-N. Guo <bai.ni.guo@hotmail.com>}{bai.ni.guo@hotmail.com}}
\urladdr{\url{http://www.researcherid.com/rid/C-8032-2013}}

\author[F. Qi]{Feng Qi}
\address[Qi]{College of Mathematics, Inner Mongolia University for Nationalities, Tongliao City, Inner Mongolia Autonomous Region, 028043, China; Department of Mathematics, College of Science, Tianjin Polytechnic University, Tianjin City, 300387, China}
\email{\href{mailto: F. Qi <qifeng618@gmail.com>}{qifeng618@gmail.com}, \href{mailto: F. Qi <qifeng618@hotmail.com>}{qifeng618@hotmail.com}, \href{mailto: F. Qi <qifeng618@qq.com>}{qifeng618@qq.com}}
\urladdr{\url{http://qifeng618.wordpress.com}}

\begin{abstract}
In the paper, the author elementarily unifies and generalizes eight identities involving the functions $\frac{\pm1}{e^{\pm t}-1}$ and their derivatives. By one of these identities, the author establishes two explicit formulae for computing Euler polynomials and two-parameter Euler polynomials, which are a newly introduced notion, in terms of Stirling numbers of the second kind.
\end{abstract}



%\thanks{Please cite this article as ``Bai-Ni Guo and Feng Qi, \textit{Explicit formulae for computing Euler polynomials in terms of Stirling numbers of the second kind}, Journal of Computational and Applied Mathematics (2015), in press; Available online at \url{http://dx.doi.org/10.1016/j.cam.2014.05.???}.''}

\maketitle

\section{Introduction}

In~\cite{exp-derivative-sum-Combined.tex}, the following eight identities were elementarily and inductively established.

\begin{thm}[{\cite[Theorems~2.1 to~2.4 and Corollaries~2.1 to~2.4]{exp-derivative-sum-Combined.tex}}]\label{exp-deriv-exp-lem}
For $k\in\mathbb{N}$, we have
\begin{align}\label{exp-deriv-exp}
\biggl(\frac1{e^t-1}\biggr)^{(k)}&=\sum_{m=1}^{k+1}\lambda_{k,m}\biggl(\frac1{e^t-1}\biggr)^m,&
\biggl(\frac1{1-e^{-t}}\biggr)^{(k)}&=\sum_{m=1}^{k+1}\mu_{k,m}\biggl(\frac1{1-e^{-t}}\biggr)^m,\\
\biggl(\frac1{1-e^{-t}}\biggr)^{(k)}&=\sum_{m=1}^{k+1}\lambda_{k,m}\biggl(\frac1{e^t-1}\biggr)^m,&
\biggl(\frac1{e^{t}-1}\biggr)^{(k)}&=\sum_{m=1}^{k+1}\mu_{k,m}\biggl(\frac1{1-e^{-t}}\biggr)^m,\\
\biggl(\frac1{1-e^{-t}}\biggr)^k&=\sum_{m=1}^{k}a_{k,m-1}\biggl(\frac1{1-e^{-t}}\biggr)^{(m-1)},&
\biggl(\frac1{e^t-1}\biggr)^k&=\sum_{m=1}^{k}b_{k,m-1}\biggl(\frac1{e^t-1}\biggr)^{(m-1)}, \label{exp-deriv-exp-last}\\
\biggl(\frac1{1-e^{-t}}\biggr)^k&=1+\sum_{m=1}^{k}a_{k,m-1}\biggl(\frac1{e^t-1}\biggr)^{(m-1)},&
\biggl(\frac1{e^t-1}\biggr)^k&=1+\sum_{m=1}^{k}b_{k,m-1}\biggl(\frac1{1-e^{-t}}\biggr)^{(m-1)},
\end{align}
where
\begin{gather}\label{lambda-stirling-relation}
\lambda_{k,m}=(-1)^k(m-1)!S(k+1,m),\quad
\mu_{k,m}=(-1)^{m-1}(m-1)!S(k+1,m),\\
a_{k,m-1}=(-1)^{m^2+1}M_{k-m+1}(k,m), \quad b_{k,m-1}=(-1)^{k-m}a_{k,m-1}, \label{a-b-m-relat}\\
\label{Mn-k+1(nk)}
M_j(k,i)=
\begin{vmatrix}
\frac1{(i-1)!}\binom{k}{i} & S(i+1,i)  & S(i+2,i)& \dotsm  &  S(i+j-1,i)\\
\frac1{i!}\binom{k}{i+1} & S(i+1,i+1) & S(i+2,i+1)& \dotsm  &  S(i+j-1,i+1)\\
\frac1{(i+1)!}\binom{k}{i+2} & 0 & S(i+2,i+2)& \dotsm  & S(i+j-1,i+2) \\
\vdots & \vdots& \vdots & \ddots &\vdots \\
\frac1{(i+j-2)!}\binom{k}{i+j-1} & 0& 0 & \dotsm  & S(i+j-1,i+j-1)
\end{vmatrix},\quad j\in\mathbb{N},
\end{gather}
and
\begin{equation}\label{Stirling-Number-dfn}
S(k,m)=\frac1{m!}\sum_{\ell=1}^m(-1)^{m-\ell}\binom{m}{\ell}\ell^{k}, \quad 1\le m\le k
\end{equation}
are Stirling numbers of the second kind which may be generated by
\begin{equation}\label{2stirling-gen-funct-exp}
\frac{(e^x-1)^k}{k!}=\sum_{n=k}^\infty S(n,k)\frac{x^n}{n!}, \quad k\in\mathbb{N}.
\end{equation}
\end{thm}

It was pointed out in~\cite[Remark~5.3]{exp-derivative-sum-Combined.tex} that the above eight identities involving the functions $\frac{\pm1}{e^{\pm t}-1}$ and their derivatives are equivalent to each other.
\par
By virtue of the first identity in~\eqref{exp-deriv-exp}, among other things, an explicit formula for computing Bernoulli numbers $B_{2k}$, which are defined by the power series expansion
\begin{equation}\label{Bernumber-dfn}
\frac{t}{e^t-1}=\sum_{i=0}^\infty B_i\frac{t^i}{i!}=1-\frac{t}2+\sum_{k=1}^\infty B_{2k}\frac{t^{2k}}{(2k)!}
\end{equation}
for $\vert t\vert<2\pi$, in terms of Stirling numbers of the second kind $S(n,k)$, was discovered in~\cite{exp-derivative-sum-Combined.tex} as follows.

\begin{thm}[{\cite[Theorem~3.1]{exp-derivative-sum-Combined.tex}}] \label{stirling-no-identity-thm}
For $k\in\mathbb{N}$, Bernoulli numbers $B_{2k}$ may be computed by
\begin{equation}\label{Bernumber-formula-eq}
B_{2k}=1+\sum_{m=1}^{2k-1}\frac{S(2k+1,m+1) S(2k,2k-m)}{\binom{2k}{m}}
-\frac{2k}{2k+1}\sum_{m=1}^{2k}\frac{S(2k,m)S(2k+1,2k-m+1)}{\binom{2k}{m-1}}.
\end{equation}
\end{thm}

\par
In~\cite{CAM-D-13-01430-Xu-Cen}, making use of Fa\`a di Bruno's formula, combinatorial techniques, and much knowledge on Bell polynomials of the second kind and Stirling numbers of the first and second kinds, the above eight identities were generalized and unified as follows.

\begin{thm}[{\cite[Theorems~3.1 and~3.2]{CAM-D-13-01430-Xu-Cen}}]\label{CAM-D-13-01430-Xu-Cen-thm}
For $\alpha,\lambda\in\mathbb{R}$,
\begin{enumerate}
\item
when $n\in\mathbb{N}$, we have
\begin{equation}\label{xu-cen-f1}
\biggl(\frac1{1-\lambda e^{\alpha t}}\biggr)^{(n)}=(-1)^{n}\alpha^n \sum_{k=1}^{n+1}(-1)^{k-1}(k-1)!S(n+1,k)\biggl(\frac1{1-\lambda e^{\alpha t}}\biggr)^k;
\end{equation}
\item
when $n\in\mathbb{N}$, we have
\begin{equation}\label{xu-cen-f2}
\biggl(\frac1{1-\lambda e^{\alpha t}}\biggr)^{n}=\frac{(-1)^{n-1}}{(n-1)!} \sum_{k=1}^n\frac{(-1)^{k-1}}{\alpha^{k-1}} s(n,k)\biggl(\frac1{1-\lambda e^{\alpha t}}\biggr)^{(k-1)},
\end{equation}
where $s(n,k)$ for $n\ge k\ge1$ denote Stirling numbers of the first kind which may be generated by
\begin{equation}\label{1stirling-gen-funct-log}
\frac{[\ln(1+x)]^m}{m!}=\sum_{k=m}^\infty s(k,m)\frac{x^k}{k!},\quad |x|<1.
\end{equation}
\end{enumerate}
\end{thm}

As a consequence of comparing the identity~\eqref{xu-cen-f2} for $\alpha=-1$ and $\lambda=1$ with the first identity in~\eqref{exp-deriv-exp-last}, it was derived in~\cite[Corollary~3.2]{CAM-D-13-01430-Xu-Cen} that, when $n\ge k$,
\begin{equation}\label{s(n-k)-M(n-k)}
s(n,k)=(-1)^{n+k^2}(n-1)!M_{n-k+1}(n,k),
\end{equation}
where $M_{n-k+1}(n,k)$ is defined by~\eqref{Mn-k+1(nk)}.
\par
Employing the identity~\eqref{xu-cen-f1} for $\alpha=1$, among other things, an explicit representation for calculating Apostol-Bernoulli numbers $\mathcal{B}_n(\lambda)$ for $n\in\{0\}\cup\mathbb{N}$, which may be defined by
\begin{equation}
\frac{t}{\lambda e^t-1}=\sum_{k=0}^\infty\mathcal{B}_k(\lambda)\frac{t^k}{k!}
\end{equation}
for $\lambda\in\mathbb{R}$ and $\vert t\vert<2\pi$, in terms of Stirling numbers of the second kind $S(n,k)$, was obtained in~\cite{CAM-D-13-01430-Xu-Cen} as follows.

\begin{thm}[{\cite[Theorem~4.1]{CAM-D-13-01430-Xu-Cen}}]
For $\lambda\ne1$ and $n\in\mathbb{N}$, we have
\begin{equation}
\mathcal{B}_n(\lambda)=(-1)^{n-1}n\sum_{k=1}^n\frac{(k-1)!}{(\lambda-1)^k}S(n,k).
\end{equation}
\end{thm}

The first aim of this paper is to supply, just basing on the first identity in~\eqref{exp-deriv-exp} and the second identity in~\eqref{exp-deriv-exp-last}, without using the abstruse Fa\`a di Bruno's formula, deep techniques in combinatorics, and any recondite knowledge on Bell polynomials and Stirling numbers of the first and second kinds, an elementary proof for the identities~\eqref{xu-cen-f1} and~\eqref{xu-cen-f2} in Theorem~\ref{CAM-D-13-01430-Xu-Cen-thm}.
\par
The second aim of this paper is to find explicit formulae for computing Euler numbers and polynomials in terms of Stirling numbers of the second kind $S(n,k)$.
\par
The third aim of this paper is to introduce a notion ``two-parameter Euler polynomials'', a generalization of the classical Euler polynomials, and establish an explicit formula for computing the newly defined two-parameter Euler polynomials in terms of Stirling numbers of the second kind $S(n,k)$.

\section{An elementary proof of Theorem~\ref{CAM-D-13-01430-Xu-Cen-thm}}

In order to elementarily prove Theorem~\ref{CAM-D-13-01430-Xu-Cen-thm}, we would like to rewrite it as Theorem~\ref{Xu-Cen-rew-thm} below.

\begin{thm}\label{Xu-Cen-rew-thm}
Let $\lambda\ne0$ and $\alpha\ne0$ be real constants and $k\in\mathbb{N}$. When $\lambda>0$ and $t\ne-\frac{\ln\lambda}\alpha$ or when $\lambda<0$ and $t\in\mathbb{R}$, we have
\begin{equation}\label{id-gen-new-form1}
\frac{\td^k}{\td t^k}\biggl(\frac1{\lambda e^{\alpha t}-1}\biggr)
=(-1)^k\alpha^k\sum_{m=1}^{k+1}{(m-1)!S(k+1,m)}\biggl(\frac1{\lambda e^{\alpha t}-1}\biggr)^m
\end{equation}
and
\begin{equation}\label{id-gen-new-form2}
\biggl(\frac1{\lambda e^{\alpha t}-1}\biggr)^{k}=\frac1{(k-1)!} \sum_{m=1}^k\frac{(-1)^{m-1}}{\alpha^{m-1}} s(k,m)\frac{\td^{m-1}}{\td t^{m-1}}\biggl(\frac1{\lambda e^{\alpha t}-1}\biggr),
\end{equation}
where $s(k,m)$ and $S(k+1,m)$ represent Stirling numbers of the first and second kinds.
\end{thm}

\begin{proof}
For $\lambda>0$, let
\begin{equation}\label{F(u)-u-def}
F(u)=\frac1{e^u-1} \quad \text{and}\quad u=u(t)=\ln\lambda+\alpha t.
\end{equation}
From $\lambda e^{\alpha t}=e^{\ln\lambda+\alpha t}=e^{u(t)}$ and $u'(t)=\alpha$, it follows that for $1<\ell<k$
\begin{gather*}
\frac{\td^k}{\td t^k}\biggl(\frac1{\lambda e^{\alpha t}-1}\biggr)
=\frac{\td^k}{\td t^k}\biggl(\frac1{e^{\ln\lambda+\alpha t}-1}\biggr)
=\frac{\td^kF(u(t))}{\td t^k}\\
=\alpha\frac{\td^{k-1}F'(u(t))}{\td t^{k-1}}
=\alpha^\ell\frac{\td^{k-\ell}F^{(\ell)}(u(t))}{\td t^{k-\ell}}
=\alpha^kF^{(k)}(\ln\lambda+\alpha t).
\end{gather*}
Combining this with the first identity~\eqref{exp-deriv-exp} yields that
\begin{gather*}
\frac{\td^k}{\td t^k}\biggl(\frac1{\lambda e^{\alpha t}-1}\biggr)
=\alpha^kF^{(k)}(\ln\lambda+\alpha t)
=\alpha^k\sum_{m=1}^{k+1}\lambda_{k,m}\biggl(\frac1{e^{\ln\lambda+\alpha t}-1}\biggr)^m\\
=\alpha^k\sum_{m=1}^{k+1}\lambda_{k,m}\biggl(\frac1{\lambda e^{\alpha t}-1}\biggl)^m
=(-1)^k\alpha^k\sum_{m=1}^{k+1}{(m-1)!S(k+1,m)}\biggl(\frac1{\lambda e^{\alpha t}-1}\biggr)^m.
\end{gather*}
The identity~\eqref{id-gen-new-form1} for $\lambda>0$ is thus proved.
\par
By the second identity in~\eqref{exp-deriv-exp-last} and the equalities in~\eqref{a-b-m-relat} and~\eqref{s(n-k)-M(n-k)}, it follows that for $1<\ell<k$
\begin{gather*}
\biggl(\frac1{\lambda e^{\alpha t}-1}\biggr)^k =\biggl(\frac1{e^{\ln\lambda+\alpha t}-1}\biggr)^k
=\biggl(\frac1{e^u-1}\biggr)^k
=\sum_{m=1}^{k}b_{k,m-1}\frac{\td^{m-1}}{\td u^{m-1}}\biggl(\frac1{e^u-1}\biggr)\\
=\sum_{m=1}^{k}\frac{b_{k,m-1}}{\alpha}\frac{\td^{m-2}}{\td u^{m-2}}\biggl[\frac{\td}{\td t}\biggl(\frac1{e^u-1}\biggr)\biggr]
=\sum_{m=1}^{k}\frac{b_{k,m-1}}{\alpha^\ell}\frac{\td^{m-\ell-1}}{\td u^{m-\ell-1}}\biggl[\frac{\td^\ell}{\td t^\ell}\biggl(\frac1{e^u-1}\biggr)\biggr]\\
=\sum_{m=1}^{k}\frac{b_{k,m-1}}{\alpha^{m-1}}\frac{\td^{m-1}}{\td t^{m-1}}\biggl(\frac1{e^{\ln\lambda+\alpha t}-1}\biggr)
=\sum_{m=1}^{k}\frac{(-1)^{k-m}a_{k,m-1}}{\alpha^{m-1}}\frac{\td^{m-1}}{\td t^{m-1}}\biggl(\frac1{\lambda e^{\alpha t}-1}\biggr)\\
=\sum_{m=1}^{k}\frac{(-1)^{k-m}(-1)^{m^2+1}M_{k-m+1}(k,m)}{\alpha^{m-1}}\frac{\td^{m-1}}{\td t^{m-1}}\biggl(\frac1{\lambda e^{\alpha t}-1}\biggr)\\
=\frac1{(k-1)!} \sum_{m=1}^k\frac{(-1)^{m-1}}{\alpha^{m-1}} s(k,m)\frac{\td^{m-1}}{\td t^{m-1}} \biggl(\frac1{\lambda e^{\alpha t}-1}\biggr).
\end{gather*}
The identity~\eqref{id-gen-new-form2} for $\lambda>0$ is thus proved.
\par
Let
\begin{equation}
H(t)=\frac1{e^t+1},\quad t\in\mathbb{R}.
\end{equation}
It is clear that $F(t)$, defined in~\eqref{F(u)-u-def}, and $H(t)$ may be regarded as composite functions of $\frac1{u}$ with $u=u(t)=e^t\pm1$ respectively and that $u'(t)=e^t$.
Further by implications of the first identity in~\eqref{exp-deriv-exp} and the second identity in~\eqref{exp-deriv-exp-last} and by induction, we may obtain that
\begin{equation}\label{exp-deriv-exp-plus}
\biggl(\frac1{e^t+1}\biggr)^{(k)}=\sum_{m=1}^{k+1}(-1)^{m-1}\lambda_{k,m}\biggl(\frac1{e^t+1}\biggr)^m \quad\text{and}\quad
\biggl(\frac1{e^{t}+1}\biggr)^k=(-1)^{k-1}\sum_{m=1}^{k}b_{k,m-1}\biggl(\frac1{e^t+1}\biggr)^{(m-1)}.
\end{equation}
\par
For $\lambda<0$, let $v=v(t)=\ln|\lambda|+\alpha t$. As the above arguments, from $\lambda e^{\alpha t}=-e^{\ln|\lambda|+\alpha t}$ and $v'(t)=\alpha$, it follows that for $1<\ell<k$
\begin{gather*}
\frac{\td^k}{\td t^k}\biggl(\frac1{\lambda e^{\alpha t}-1}\biggr)
=-\frac{\td^k}{\td t^k}\biggl(\frac1{e^{\ln|\lambda|+\alpha t}+1}\biggr)
=-\frac{\td^kH(v(t))}{\td t^k}
=-\alpha\frac{\td^{k-1}H'(v(t))}{\td t^{k-1}}\\
=-\alpha^\ell\frac{\td^{k-\ell}H^{(\ell)}(v(t))}{\td t^{k-\ell}}
=-\alpha^kH^{(k)}(v(t))
=-\alpha^kH^{(k)}(\ln|\lambda|+\alpha t).
\end{gather*}
Combining this with the first identity in~\eqref{exp-deriv-exp-plus} and the first formula in~\eqref{lambda-stirling-relation} gives that
\begin{gather*}
\frac{\td^k}{\td t^k}\biggl(\frac1{\lambda e^{\alpha t}-1}\biggr)
=-\alpha^kH^{(k)}(\ln|\lambda|+\alpha t)
=-\alpha^k\sum_{m=1}^{k+1}(-1)^{m-1}\lambda_{k,m}\biggl(\frac1{e^{\ln|\lambda|+\alpha t}+1}\biggl)^m\\
=-\alpha^k\sum_{m=1}^{k+1}(-1)^{m-1}\lambda_{k,m}\biggl(\frac1{-\lambda e^{\alpha t}+1}\biggr)^m
=(-1)^k\alpha^k\sum_{m=1}^{k+1}{(m-1)!S(k+1,m)} \biggl(\frac1{\lambda e^{\alpha t}-1}\biggr)^m.
\end{gather*}
The identity~\eqref{id-gen-new-form1} for $\lambda<0$ is thus proved.
\par
By the second identity in~\eqref{exp-deriv-exp-plus} and the equalities in~\eqref{a-b-m-relat} and~\eqref{s(n-k)-M(n-k)}, it follows that for $\lambda<0$
\begin{gather*}
\biggl(\frac1{\lambda e^{\alpha t}-1}\biggr)^k
=\biggl(\frac1{-e^{\ln|\lambda|+\alpha t}-1}\biggr)^k
=(-1)^k\biggl(\frac1{e^{v}+1}\biggr)^k
=-\sum_{m=1}^{k}b_{k,m-1}\frac{\td^{m-1}}{\td v^{m-1}}\biggl(\frac1{e^v+1}\biggr)\\
=-\sum_{m=1}^{k}\frac{b_{k,m-1}}{\alpha^{m-1}}\frac{\td^{m-1}}{\td t^{m-1}} \biggl(\frac1{e^{\ln|\lambda|+\alpha t}+1}\biggr)
=\sum_{m=1}^{k}\frac{(-1)^{k-m}a_{k,m-1}}{\alpha^{m-1}} \frac{\td^{m-1}}{\td t^{m-1}}\biggl(\frac1{\lambda e^{\alpha t}-1}\biggr)\\
=\sum_{m=1}^{k}\frac{(-1)^{k-m}(-1)^{m^2+1}M_{k-m+1}(k,m)}{\alpha^{m-1}}\frac{\td^{m-1}}{\td t^{m-1}}\biggl(\frac1{\lambda e^{\alpha t}-1}\biggr)\\
=\frac1{(k-1)!} \sum_{m=1}^k\frac{(-1)^{m-1}}{\alpha^{m-1}} s(k,m)\frac{\td^{m-1}}{\td t^{m-1}} \biggl(\frac1{\lambda e^{\alpha t}-1}\biggr).
\end{gather*}
The identity~\eqref{id-gen-new-form2} for $\lambda<0$ is thus proved.
The proof of Theorem~\ref{Xu-Cen-rew-thm} is complete.
\end{proof}

\section{Explicit formulae for Euler numbers and polynomials}

It is general knowledge that Euler numbers and polynomials may be generated as follows.

\begin{dfn}[{\cite[p.~804]{abram}}]\label{euldfn1}
For $x\in\mathbb{R}$ and $k\in\{0\}\cup\mathbb{N}$, Euler numbers $E_k$ and Euler polynomials $E_k(x)$ are respectively defined by the power expansions
\begin{equation}\label{wangdef}
\frac{2e^{t/2}}{e^t+1}
=\sum_{k=0}^\infty\frac{E_{k}}{k!}\biggl(\frac{t}2\biggr)^{k} \quad \text{and}\quad
\frac{2e^{xt}}{e^t+1}=\sum_{k=0}^\infty\frac{E_k(x)}{k!}t^k
\end{equation}
which converge uniformly with respect to $t\in(-\pi,\pi)$.
\end{dfn}

By definition, it is clear that
\begin{equation}\label{Bernouli-No-Polyn}
E_n=2^nE_n\biggl(\frac12\biggr).
\end{equation}
\par
Since the generating function $\frac{2e^{t/2}}{e^t+1}$ of Euler numbers $E_k$ in~\eqref{euldfn1} is even on $\mathbb{R}$, then $E_{2k-1}=0$ for all $k\in\mathbb{N}$.

\begin{thm}\label{euler-poly-stirling-thm}
For $n\in\mathbb{N}$, Euler polynomials $E_n(x)$ may be calculated by
\begin{equation}\label{Euler-Stirling-eq}
E_n(x)=\sum_{k=0}^n(-1)^{n-k}\binom{n}{k}\Biggl[\sum_{\ell =1}^{n-k+1} \frac{(-1)^{\ell -1}(\ell -1)!}{2^{\ell -1}}S(n-k+1,\ell )\Biggr]x^{k},
\end{equation}
where $S(n-k+1,\ell)$ are Stirling numbers of the second kind.
Consequently, Euler numbers $E_{2n}$ for $n\in\mathbb{N}$ may be calculated by
\begin{equation}\label{euler-no-stirling-eq}
E_{2n}=4^n\sum_{k=0}^{2n}\Biggl[\sum_{\ell=1}^{{2n}-k+1} \frac{(-1)^{\ell-1}(\ell-1)!}{2^{\ell-1}}S({2n}-k+1,\ell)\Biggr]\frac{(-1)^{k}}{2^{k}}\binom{{2n}}{k}.
\end{equation}
Moreover, Stirling numbers $S(n,k)$ satisfy
\begin{equation}\label{euler-no-stirling=0}
\sum_{k=0}^{2n-1}\Biggl[\sum_{\ell=1}^{2n-k} \frac{(-1)^{\ell-1}(\ell-1)!}{2^{\ell-1}}S(2n-k,\ell)\Biggr] \frac{(-1)^{k}}{2^{k}}\binom{2n-1}{k}=0.
\end{equation}
\end{thm}

\begin{proof}
By Leibniz's theorem for differentiation and the first identity in~\eqref{exp-deriv-exp-plus}, we have
\begin{gather*}
\frac{\td^n}{\td t^n}\biggl(\frac{2e^{xt}}{e^t+1}\biggr)=2\sum_{i=0}^n\binom{n}{i} \frac{\td^ie^{xt}}{\td t^i} \frac{\td^{n-i}}{\td t^{n-i}}\biggl(\frac1{e^t+1}\biggr) \\
=2e^{xt}\sum_{i=0}^n\binom{n}{i}x^{i}\biggl(\frac1{e^t+1}\biggr)^{(n-i)}
=2e^{xt}\sum_{i=0}^n\sum_{j=1}^{n-i+1}\binom{n}{i}(-1)^{j-1} \lambda_{n-i,j}\biggl(\frac1{e^t+1}\biggr)^jx^{i}.
\end{gather*}
Combining this with the $n$-th differentiation on both sides of the second generating function in~\eqref{wangdef} reveal that
\begin{equation*}
\sum_{k=n}^\infty E_k(x)\frac{t^{k-n}}{(k-n)!}
=2e^{xt}\sum_{i=0}^n\sum_{j=1}^{n-i+1}\binom{n}{i}(-1)^{j-1} \lambda_{n-i,j}\biggl(\frac1{e^t+1}\biggr)^jx^{i}.
\end{equation*}
Further taking $t\to0$ and employing the first equality~\eqref{lambda-stirling-relation} give
\begin{equation*}
E_n(x)=\sum_{i=0}^n\sum_{j=1}^{n-i+1}\frac{(-1)^{j-1}}{2^{j-1}}\binom{n}{i}\lambda_{n-i,j}x^{i}
=(-1)^n\sum_{i=0}^n(-1)^{i}\binom{n}{i}\Biggl[\sum_{j=1}^{n-i+1} \frac{(-1)^{j-1}}{2^{j-1}}(j-1)!S(n-i+1,j)\Biggr]x^{i}.
\end{equation*}
The proof of the formula~\eqref{Euler-Stirling-eq} is complete.
\par
Replacing $n$ by $2n$ and $x$ by $\frac12$ in~\eqref{Euler-Stirling-eq} and using~\eqref{Bernouli-No-Polyn} produce the formula~\eqref{euler-no-stirling-eq}.
\par
The equality~\eqref{euler-no-stirling=0} follows from the property $E_{2k-1}=0$ for all $k\in\mathbb{N}$ and the relation~\eqref{Bernouli-No-Polyn}. The proof of Theorem~\ref{euler-poly-stirling-thm} is complete.
\end{proof}

\section{Two-parameter Euler polynomials and their explicit formula}

Euler numbers and polynomials mentioned in the above section may be generalized as follows.

\begin{dfn}\label{euldfn-lambda}
For $x\in\mathbb{R}$, $\alpha\ne0$, and $\lambda>0$, the quantity $E_k(x;\alpha,\lambda)$ generated by
\begin{equation}\label{wangdef-parameter}
\frac{2e^{xz}}{\lambda e^{\alpha z}+1}=\sum_{k=0}^\infty E_k(x;\alpha,\lambda)\frac{z^k}{k!}, \quad |\alpha z+\ln\lambda|<\pi
\end{equation}
are called two-parameter Euler polynomials.
\end{dfn}

It is easy to deduce that two-parameter Euler polynomials $E_k(x;\alpha,\lambda)$ satisfy
\begin{equation}
E_k(x;1,1)=E_k(x) \quad \text{and}\quad E_k(x;\alpha,\lambda)=\alpha^kE_k\biggl(\frac{x}\alpha;1,\lambda\biggr) =x^kE_k\biggl(1;\frac\alpha{x},\lambda\biggr).
\end{equation}
This shows that the introduction of two-parameter Euler polynomials $E_k(x;\alpha,\lambda)$ is not trivial.

\begin{thm}\label{2meter-euler-formula-thm}
For $x\in\mathbb{R}$, $\alpha\ne0$, $\lambda>0$, and $n\in\{0\}\cup\mathbb{N}$, two-parameter Euler polynomials $E_n(x;\alpha,\lambda)$ may be computed in terms of Stirling numbers of the second kind $S(n,k)$ by
\begin{equation}\label{2meter-euler-formula}
E_n(x;\alpha,\lambda)=2\sum_{k=0}^n(-\alpha)^{n-k}\binom{n}{k}\Biggl[\sum_{m=1}^{n-k+1} (-1)^{m-1}(m-1)!S(n-k+1,m) \biggl(\frac1{\lambda+1}\biggr)^m\Biggr]x^{k}.
\end{equation}
\end{thm}

\begin{proof}
In light of Leibniz's theorem for differentiation and the formula~\eqref{id-gen-new-form1}, we have
\begin{align*}
\frac{\td^n}{\td t^n}\biggl(\frac{2e^{xt}}{\lambda e^{\alpha t}+1}\biggr)
&=2\sum_{k=0}^n\binom{n}{k}\frac{\td^k}{\td t^k}\biggl(\frac1{\lambda e^{\alpha t}+1}\biggr) \frac{\td^{n-k}e^{xt}}{\td t^{n-k}}\\
&=2e^{xt}\sum_{k=0}^n\sum_{m=1}^{k+1}\binom{n}{k}x^{n-k} (-1)^{k+m-1}\alpha^k(m-1)!S(k+1,m) \biggl(\frac1{\lambda e^{\alpha t}+1}\biggr)^m.
\end{align*}
As a result, it follows that
\begin{equation*}
\sum_{k=n}^\infty E_k(x;\alpha,\lambda)\frac{t^{k-n}}{(k-n)!}
=2e^{xt}\sum_{k=0}^n\sum_{m=1}^{k+1}\binom{n}{k}x^{n-k} (-1)^{k+m-1}\alpha^k(m-1)!S(k+1,m) \biggl(\frac1{\lambda e^{\alpha t}+1}\biggr)^m.
\end{equation*}
Further taking the limit $t\to0$ leads to
\begin{equation*}
E_n(x;\alpha,\lambda)=2\sum_{k=0}^n\sum_{m=1}^{k+1}\binom{n}{k}x^{n-k} (-1)^{k+m-1}\alpha^k(m-1)!S(k+1,m) \biggl(\frac1{\lambda+1}\biggr)^m.
\end{equation*}
The proof of Theorem~\ref{2meter-euler-formula-thm} is complete.
\end{proof}

\begin{rem}
The second equality in~\eqref{a-b-m-relat} corrects equations~(2.23) and~(2.26) in~\cite[p.~573]{exp-derivative-sum-Combined.tex}.
\end{rem}

\begin{rem}
The functions $\frac{\pm1}{e^{\pm t}-1}$ and their derivatives have also been investigated in the paper~\cite{Exp-Diff-Ratio-Wei-Guo.tex} in a different direction.
\end{rem}

\begin{rem}
In~\cite{Filomat-36-73-1.tex, 1st-Sirling-Number-2012.tex}, several formulae for computing Stirling numbers of the first kind were discovered. By these formulae, some properties of Stirling numbers of the first kind were found in~\cite{Filomat-36-73-1.tex, 1st-Sirling-Number-2012.tex} and closely related references therein.
\end{rem}

\begin{rem}
After completing this paper, the author discovers that the first identity in~\eqref{exp-deriv-exp} and the special case $\lambda=\alpha=1$ of the identity~\eqref{id-gen-new-form2} were listed, but without proof, in~\cite[p.~559]{GKP-Concrete-Math}.
\end{rem}

\begin{rem}
This paper is a slightly modified version of the preprint~\cite{Eight-Identy-More.tex}.
\end{rem}

\subsection*{Acknowledgements}
The authors thank the anonymous referee for his/her helpful corrections to and valuable comments on the original version of this paper.
\par
The second author was partially supported by the NNSF under Grant No. 11361038 of China.

\end{document}